\documentclass[reqno]{amsart}
\usepackage{float}
\usepackage{amsmath}
\usepackage{graphicx}
\usepackage{latexsym}
\usepackage{amsfonts}
\usepackage{amssymb}
\theoremstyle{plain}
\newtheorem{theorem}{Theorem}
\newtheorem{corollary}[theorem]{Corollary}
\newtheorem{lemma}[theorem]{Lemma}
\newtheorem{proposition}[theorem]{Proposition}
\theoremstyle{definition}

\newtheorem{remark}[theorem]{Remark}
\newtheorem*{remark*}{Remark}

\begin{document}\title[Upper bounds for the maximum of a random walk with negative drift]
{Upper bounds for the maximum of a random walk with negative drift}

\author[Kugler]{Johannes Kugler}
\address{Mathematical Institute, University of Munich,
Theresienstrasse 39, D-80333, Munich, Germany}
\email{kugler@math.lmu.de}
\author[Wachtel]{Vitali Wachtel}
\address{Mathematical Institute, University of Munich,
Theresienstrasse 39, D-80333, Munich, Germany}
\email{wachtel@math.lmu.de}

\date{\today }

\begin{abstract}
 Consider a random walk $S_n=\sum_{i=0}^n X_i$ with negative drift. 
 This paper deals with upper bounds for the maximum $M=\max_{n\ge 1}S_n$ of this random walk in different settings of power moment existences.
 As it is usual for deriving upper bounds, we truncate summands. 
 Therefore we use an approach of splitting the time axis by stopping times into intervals of random but finite length and then choose a level of truncation on each interval. 
 Hereby we can reduce the problem of finding upper bounds for $M$ to the problem of finding upper bounds for $M_\tau=\max_{n\le \tau}S_n$.
 In addition we test our inequalities in the heavy traffic regime in the case of regularly varying tails.
\end{abstract}

\thanks{Supported by the DFG}
\maketitle

\section{Introduction and statement of results}
Let $\{S_n,n\geq0\}$ denote the random walk with increments $X_i$, that is,
$$
S_0:=0,\ S_n:=\sum_{i=1}^n X_i,\ n\geq1.
$$
We shall assume that $X_1,X_2,\ldots$ are independent copies of a random variable 
$X$ with distribution function $F$ and $a:=-\mathbf{E}[X]>0$. The random walk $S_n$ drifts to $-\infty$ and the total 
maximum $M:=\max_{k\geq0} S_k$ is finite almost surely.
The random variable $M$ plays a crucial role in a number of applications. For example, its distribution
coincides with the stationary distribution of the queue-length in simple queueing systems. Another 
important application comes from the insurance mathematics: Under some special restrictions on $X$ 
the quantity $\mathbf{P}(M>u)$ is equal to the ruin probability in the so-called Poisson model.

The tail-behaviour of $M$ has been studied extensively in the literature. 
The first result goes back, apparently, to Cramer and Lundberg (see, for example, Asmussen \cite{A00}): If
\begin{equation}
\label{Cramer}
\mathbf{E}[e^{h_0X}]=1\quad\text{for some }h_0>0,
\end{equation}
and, in addition, $\mathbf{E}[Xe^{h_0X}]<\infty$, then there exists a constant $c_0\in(0,1)$ such that
\begin{equation}
\label{Cramer-Lundberg}
\mathbf{P}(M>x)\sim c_0e^{-h_0x}\quad\text{as }x\to\infty.
\end{equation}
The case $\mathbf{E}[Xe^{h_0X}]=\infty$ has been considered recently by Korshunov \cite{K05}.

If (\ref{Cramer}) is not fulfilled, then one assumes that the distribution of $X$ is regular in some sense. To specify what regular means we recall some definitions and known properties. For their proofs we refer to Asmussen \cite{A00}.
Consider a distribution function $B$ on $\mathbb{R}$ and let $\overline{B}(x)=1-B(x)$ be the right tail of $B$. A distribution function $B$ with support $\mathbb{R}_+$ is called subexponential, if $\overline{B}(x)>0$ for all $x$ and
\begin{equation}
\label{subexponential}
 \lim_{x\to \infty} \frac{\overline{B^{*n}}(x)}{\overline{B}(x)}=n
\end{equation}
for all $n\ge 2$, where $\overline{B^{*n}}(x)$ is the $n$-fold convolution of $B$ with itself. For the subexponentiality it is sufficient to verify the equation 
(\ref{subexponential}) in the case $n=2$. All subexponential distributions are heavy-tailed, i.e. $\mathbf{E}[\exp(\epsilon X)] = \infty$ for all $\epsilon>0$,
hence subexponential distributions do not satisfy (\ref{Cramer}).
If (\ref{Cramer}) is not fulfilled, the most classical result for the asymptotics of $M$ is due to Veraverbeke \cite{V77}, who showed that if the integrated tail $G(x):=\int_x^\infty \overline{F}(u)du$
is subexponential, then
\begin{equation}\label{Veraverbeke}
\mathbf{P}(M>x)\sim\frac{1}{a}G(x)\quad\text{as }x\to\infty.
\end{equation}

In many situations one needs non-asymptotic properties of the distribution of $M$. Since the exact form of that
distribution is known in some special cases only, good estimates are required. Under condition (\ref{Cramer}) one
has for all $x>0$ the inequality
\begin{equation}
\label{Cramer-Lundberg2}
\mathbf{P}(M>x)\leq e^{-h_0x}.
\end{equation}

In the case when (\ref{Cramer})
is not fulfilled, upper bounds for $\mathbf{P}(M>x)$ have been derived by Kalashnikov \cite{K99} and by  Richards \cite{R09}. 
The approach in these papers is based on the representation of $M$ as a geometric sum of independent random variables:
\begin{equation}
\label{GS}
\mathbf{P}(M>x)=\sum_{k=0}^\infty q(1-q)^k\mathbf{P}(\chi_1^++\chi_2^+\ldots+\chi_k^+>x),
\end{equation}
where $\{\chi^+_k\}$ are independent random variables and $q=\mathbf{P}(M=0)$.
The main difficulty in this approach is the fact that one has to know the distribution of $\chi^+_k$ and the parameter $q$.
In some special cases this information can be obtained from the initial data. But in general one has to obtain appropriate estimates
for $q$ and $\mathbf{P}(\chi^+_1>x)$.

The main purpose of the present paper is to derive upper bounds for $\mathbf{P}(M>x)$ assuming the existence of power moments of $X$
only. Thereby we want to avoid the representation via geometric sums and use a supermartingale-construction instead.

As it is usual for deriving upper bounds, we are going to truncate summands and to use inequalities, which are based on truncated
exponential moments. But the problem is that we have infinitely many $X_i$'s. So we can not truncate all of them at the same 
level. Thus, we have to split the time axis into intervals of finite length and then choose a level of truncation on each 
of these intervals. One can take, for example, a deterministic strictly increasing sequence $k_n$ with $k_0=0$ and consider 
the intervals $I_n:=(k_n,k_{n+1}]$:
\begin{align}
\label{Split}
\nonumber
\mathbf{P}\left(M\geq x\right)&=\mathbf{P}\left(\bigcup_{k\geq0}\{S_k\geq x\}\right)
\leq\sum_{n=0}^\infty\mathbf{P}\left(\bigcup_{k\in I_n}\{S_k\geq x\}\right)\\
&\leq \sum_{n=0}^\infty\mathbf{P}\left(\max_{k\leq k_{n+1}}(S_k-ka)\geq x-k_{n}a\}\right).
\end{align}
Now, one can apply the Fuk-Nagaev inequalities, see \cite{N79}, to every probability in the last line.
It is clear that replacing $\sup_{k\in I_n}(S_k-ka)$ by $\sup_{k\leq k_{n+1}}(S_k-ka)$ is not too rough
if and  only if $k_{n+1}$ and $k_{n+1}-k_n$ are comparable. Thus, one has to take $k_n$ exponentially growing.
Using this approach with $k_n=x2^n$, Borovkov \cite{B76} obtained a version of the Markov inequality for $M$.

Our strategy is quite different and consists in splitting $[0,\infty)$ into random intervals defined by a sequence of stopping times.
More precisely, we introduce the stopping time 
$$
 \tau_z := \min\{k\ge 0:S_k\le -z\}, \quad z\le x.
$$
Let $M_\tau = \max_{1\le k\le \tau_z} S_{k}$. We split the tail probability
\begin{equation}
\label{eq1.0}
 \mathbf{P}\left(M > x \right) 
 \le \mathbf{P}\left(M_{\tau_z} > x \right)+\mathbf{P}\left(\max_{k\ge \tau_z} S_{k} > x \right)
\end{equation}
and can consider the continuation of the process $(S_k)$ beyond $\tau_z$ as a probabilistic replica of the entire process. 
By $S_{\tau_z}\le -z \ a.s.$ follows
$$
	\mathbf{P}\left(\max_{k\ge \tau_z} S_{k} > x \right)
	\le \mathbf{P}\left(M > x+z \right).
$$
As a result, we have
$$
	\mathbf{P}\left(M > x \right) 
	\le \mathbf{P}\left(M_\tau > x \right)
		+ \mathbf{P}\left(M > x+z \right),
$$
and inductively we conclude
\begin{equation}
\label{eq1.1}
	\mathbf{P}\left(M > x \right)
	\le \sum_{j=0}^{\infty} \mathbf{P}\left(M_\tau > x+j z \right).
\end{equation}
It is worth mentioning that the difference between (\ref{Split}) and (\ref{eq1.1}) is the same as between
Riemann and Lebesgue integrals: We do not fit the random walk $S_n$ into a fixed splitting of the time, but choose the splitting depending 
on the paths of the random walk.

A decomposition similar to (\ref{eq1.0}) has been used by Denisov \cite{D05} for deriving the asymptotics of $\mathbf{P}(M_{\tau_0}>x)$ 
from that of $\mathbf{P}(M\in [x,x-S_{\tau_0}))$. In the present paper we use the opposite approach: We obtain estimates for $\mathbf{P}(M>x)$
from the ones for $\mathbf{P}(M_{\tau_z}>x)$.

We now state our results on $M_\tau$.
\begin{theorem}
\label{T1}
Assume that $A_t:=\mathbf{E}[|X|^t]<\infty$ for some $t\in(1,2]$. For all $y$ satisfying $y^{t-1}\ge(e-1)A_t a^{-1}$
we have the following inequality:
\begin{align}\label{T1.1}
\nonumber
\mathbf{P}\left(M_\tau>x\right)&\leq 
\frac{A_t^{x/y}}{a^{x/y-1}}\mathbf{E}[\tau_z] y^{-1-(t-1)x/y}\log\left(1+ay^{t-1}/A_t\right)\\
&\hspace{1.5cm}+\left(1+\frac{A_t^{x/y}}{a^{x/y}}y^{-(t-1)x/y}\right)\mathbf{E}[\tau_z]\mathbf{P}(X>y).
\end{align}
\end{theorem}
\begin{remark}\label{R1}
We show in the proof that (\ref{T1.1}) remains true, if one replaces $a$ and $A_t$ by
$-\mathbf{E}[X,|X|\leq y]$ and $A_{t}(y)=\mathbf{E}[|X|^t,|X|\leq y]$ respectively. In this case the restriction
$y^{t-1}>(e-1)a^{-1}A_t$ should be replaced by $\mathbf{E}[X,|X|\leq y]<0$. The use of truncated
moments is more convenient in theoretical applications, but for deriving concrete estimates for $M$ it is easier to use full moments.
\hfill$\diamond$
\end{remark}

Fix $\alpha\in(0,1)$ and put $\beta=1-\alpha$.
\begin{theorem}
\label{T2}
Assume $\mathbf{Var}(X)<\infty$ and $A_{t,+}:=\mathbf{E}[X^t,X>0]<\infty$ for some $t>2$. 
\begin{itemize}
\item[(i)] If $y$ satisfies the condition
$$
\frac{2\alpha a}{e^t \mathbf{Var}(X)}\leq \frac{1}{y}\log\left(1+\frac{\beta a}{A_{t,+}}y^{t-1}\right),
$$
then
\begin{align}\label{T2.1}
\nonumber
\mathbf{P}\left(M_\tau>x\right)&\leq \frac{2\alpha a^2}{e^t \mathbf{Var}(X)}\mathbf{E}[\tau_z]
\left(\exp\left\{\frac{2\alpha a x}{e^t \mathbf{Var}(X)}\right\}-1\right)^{-1}\\
&\hspace{0.3cm}+\left(1+\left(\exp\left\{\frac{2\alpha a x}{e^t \mathbf{Var}(X)}\right\}-1\right)^{-1}\right)\mathbf{E}[\tau_z]\mathbf{P}(X>y).
\end{align}
\item[(ii)] If $y$ satisfies the condition
\begin{equation}
\label{T2.2.1}
\frac{2\alpha a}{e^t \mathbf{Var}(X)}\geq \frac{1}{y}\log\left(1+\frac{\beta a}{A_{t,+}}y^{t-1}\right),
\end{equation}
then
\begin{align}\label{T2.2}
\nonumber
\mathbf{P}\left(M_\tau>x\right)&\leq 
\frac{A_{t,+}^{x/y}\beta^{-x/y}}{a^{x/y-1}}\mathbf{E}[\tau_z] y^{-1-(t-1)x/y}\log\left(1+\beta a y^{t-1}/A_{t,+}\right)\\
&\hspace{1cm}+\left(1+\frac{A_{t,+}^{x/y}\beta^{-x/y}}{a^{x/y}}y^{-(t-1)x/y}\right)\mathbf{E}[\tau_z]\mathbf{P}(X>y).
\end{align}
\end{itemize}
\end{theorem}
\begin{remark}\label{R2}
Analogously to Theorem \ref{T1} one can replace $\mathbf{Var}(X)$ and $A_{t,+}$ by the corresponding truncated expectations 
$B^2(-\infty,x)=\mathbf{E}[X^2,X\le y]$ and $A_{t,+}(y)=\mathbf{E}[X^t,X\in(0,y]]$ respectively in Theorem \ref{T2}.
\hfill$\diamond$
\end{remark}
\begin{corollary}
\label{C1}
Assume that $\mathbf{P}(|X|>x)=L(x)x^{-r}$ for some $r>1$ and 
$$
\mathbf{P}(X>x))/\mathbf{P}(|X|>x)\to p\in(0,1)\quad\text{ as }x\to\infty.
$$
Then, it follows from (\ref{T1.1}) and (\ref{T2.2}) that
$$
\limsup_{x\to\infty}\frac{\mathbf{P}\left(M_\tau>x\right)}{\mathbf{P}(X>x)}\leq\mathbf{E}[\tau_z]
$$
for every $z>0$.
\end{corollary}
But it follows from the results of Asmussen \cite{A98} (see also Denisov \cite{D05}), that
$$
\lim_{x\to\infty}\frac{\mathbf{P}\left(M_\tau>x\right)}{\mathbf{P}(X>x)}=\mathbf{E}[\tau_z]
$$ 
under the condition that the tail of $F$ is regularly varying. This means that the inequalities 
(\ref{T1.1}) and (\ref{T2.2}) are asymptotically precise in the case of regularly varying tails.

In all these inequalities we have $\mathbf{E}[\tau_z]$ on the right hand side. It is really hard to get an exact expression 
for this value via initial data, but there are good upper bounds in the literature: 
Since $\mathbf{E}[\tau_z]< \infty$ (see, for example, Feller \cite{F71}) by Wald's Identity,
\begin{equation}
\label{tau}
\mathbf{E}[\tau_z] = \frac{z+\mathbf{E}[R_z]}{a},
\end{equation}
where $R_z = -z-S_{\tau_z}$ denotes the overshoot in $\tau_z$. 
Hence, we get upper bounds for $\mathbf{E}[\tau_z]$ by the inequality of Lorden \cite{L70}: 
For $\mathbf{E}[X]<0$ and $\mathbf{E}[(X^-)^2]< \infty$,
\begin{equation}
\label{lorden}
\mathbf{E}[R_z]\le \frac{\mathbf{E}[(X^-)^2]}{a}
\end{equation}
and the one from Mogul'skii \cite{M73}: For $\mathbf{E}[X]\le 0$ and $\mathbf{E}[|X|^3]< \infty$,
\begin{equation}
\label{mogulskii}
\mathbf{E}[R_z]\le A \frac{3}{2} \frac{\mathbf{E}[|X|^3]}{\mathbf{E}[X^2]},
\end{equation}
where $A$ is a certain constant, $A \le 2$.
The disadvantage of these bounds is, that we have to assume the existence of the second or even the third moment.
We give another bound, which only requires the finiteness of the moment of order $t$, $t \in (1,2]$.
\begin{proposition}
\label{P1}
Assume that $A_{t,-}:=\mathbf{E}[(X^-)^t]<\infty$ for some $t\in(1,2]$, then, for every $z>0$,
\begin{equation}
\label{P1.1}
\mathbf{E}[R_z]
\le \frac{t^{t/(t-1)}A_{t,-}^{1/(t-1)}}{(t-1)a^{t/(t-1)}}
\left(\mathbf{E}[-X,X<0]+\frac{z^{2-t}}{t} A_{t,-}\right).
\end{equation}
\end{proposition}
Combining (\ref{tau}) with (\ref{lorden}), (\ref{mogulskii}) or (\ref{P1.1}) we obtain upper bounds for $\mathbf{E}[\tau_z]$.
Plugging these bounds into the inequalities in Theorems \ref{T1} and \ref{T2} we get bounds for $\mathbf{P}(M_\tau>x)$, which contain information on $X$ only.
So, they can be used for concrete calculations.

We now come back to the global maximum.
\begin{theorem}
\label{T3} 
Fix some $\theta\in(0,1)$ and define 
$$
c_1 := \frac{3 A_t^{1/\theta}\theta^{-(t-1)/\theta}}{(t-1)a^{1/\theta-1}}, \quad
c_2 := \frac{3 A_{t,+}^{1/\theta}\theta^{-(t-1)/\theta}}{(t-1)a^{1/\theta-1}}.
$$
\begin{itemize}
\item[(i)] Assume that $A_t<\infty$ for some $t\in(1,2]$. Then, for every 
$x$ satisfying $x^{t-1}\ge\theta^{1-t}(e-1)A_t a^{-1}$ and $x\ge z(t-1)\theta^{-1}$, we have
\begin{align}\label{T3.1}
\nonumber
\mathbf{P}(M>x)&\leq 
c_1 \frac{\mathbf{E}[\tau_z]}{z}\log\left(1+\frac{a\theta^{t-1}x^{t-1}}{A_t}\right) x^{-(t-1)/\theta}
\\
&\hspace{0.3cm}+ \left(1+\left(\frac{A_t}{\theta^{t-1}ax^{t-1}}\right)^{1/\theta}\right)\mathbf{E}[\tau_z]
\left(\frac{1}{\theta z}G(\theta x)+\mathbf{P}(X>\theta x)\right).
\end{align}
\item[(ii)] Assume that $\mathbf{Var}(X)<\infty$ and $A_{t,+}<\infty$ for some $t>2$.
Then, for every $x$ satisfying (\ref{T2.2.1}) for $y=\theta x$, $x^{t-1}\ge\theta^{1-t}(e^{\theta}-1)A_{t,+} \beta^{-1}a^{-1}$ and $x\ge z(t-1)\theta^{-1}$, we have
\begin{align}\label{T3.2}
\nonumber
\mathbf{P}(M>x)&\leq c_2 \beta^{-1/\theta} \frac{\mathbf{E}[\tau_z]}{z}\log\left(1+\frac{\beta a\theta^{t-1}x^{t-1}}{A_{t,+}}\right) x^{-(t-1)/\theta}\\
&\hspace{0.3cm}+ \left(1+\left(\frac{A_{t,+}}{\beta\theta^{t-1}ax^{t-1}}\right)^{1/\theta}\right)\mathbf{E}[\tau_z]
\left(\frac{1}{\theta z}G(\theta x)+\mathbf{P}(X>\theta x)\right).
\end{align}
\end{itemize}
\end{theorem}
\begin{corollary}
 If the assumptions of Corollary \ref{C1} hold, then it follows from Theorem \ref{T3} that
 $$
  \limsup_{x\rightarrow \infty} \frac{\mathbf{P}(M>x)}{G(x)}\le \frac{\mathbf{E}[\tau_z]}{z}\theta^{-r}.
 $$
\end{corollary}
Since the left-hand side does not depend on $\theta$ and $z$, we can let $\theta \rightarrow 1$ and $z \rightarrow \infty$.
Noting that each of (\ref{lorden}) and (\ref{mogulskii}) combined with (\ref{tau}) yields
$$
 \frac{\mathbf{E}[\tau_z]}{z} \rightarrow \frac{1}{a} \quad \text{as } z\rightarrow \infty ,
$$
we conclude 
$$
  \limsup_{x\rightarrow \infty} \frac{\mathbf{P}(M>x)}{G(x)}\le \frac{1}{a}.
$$
Comparing this with (\ref{Veraverbeke}), we see that the inequalities in Theorem \ref{T3} are asymptotically precise. 
This even remains valid, if we bound $\mathbf{E}[\tau_z]$ in these inequalities by combining (\ref{lorden}) or (\ref{mogulskii}) with (\ref{tau}).

The reason why we are able to obtain asymptotically precise bounds is, that we may choose $z$ arbitrary large.
That possibility seems to be a quite important advantage of our method compared to geometric sums.  If the distribution
of $\chi^+_1$ is subexponential, then it follows easily from (\ref{GS}) that
$$
\mathbf{P}(M > x)\sim\left(\frac{1}{q}-1\right)\mathbf{P}(\chi^+_1 > x)\quad\text{as }x\to\infty.
$$
Therefore, in order to obtain an upper bound for the maximum we need to control the quantity $1/q$. It is well known that
$1/q=\mathbf{E}[-S_{\tau_0}]=\mathbf{E}[R_0].$ Thus, we may apply (\ref{lorden}), (\ref{mogulskii}) 
or (\ref{P1.1}) with $z=0$. But corresponding inequalities for $M$ will not be asymptotically precise. Summarising, the approach via geometric sums can only lead to asymptotically precise bounds if $q$ is known.

We next test our inequalities in the heavy-traffic regime. Let $\{S^{(a)},a\geq0\}$ be a family of random walks with $\mathbf{E}[X^{(a)}]=-a$.
We shall assume that $X^{(a)}=X^{(0)}-a$ for all $a>0$. Let $M^{(a)}$ denote the corresponding maximum. It is known that if
$X^{(0)}$  belongs to the domain of attraction of a stable law, then there exists a regularly varying function 
$g(a)$ such that $g(a)M^{(a)}$ converges weakly as $a\to0$. It turns out that our inequalities may be applied to large deviation problems in
the heavy-traffic convergence mentioned above. More precisely, they give asymptotically precise bounds for the probabilities
$\mathbf{P}(M^{(a)}>x_a)$ if $x_a\gg 1/g(a)$.
In the case of $\sigma^2:=\mathbf{Var}(X^{(0)})$ being finite, one has $g(a)=a$ and the weak limit of $aM^{(a)}$ 
is the exponential distribution with parameter $2/\sigma^2$.
\begin{theorem}
\label{T4}
Assume that $\sigma^2<\infty$ and the right tail of the distribution function of $X^{(0)}$ is regularly varying with index $r>2$,
that is, $\mathbf{P}(X^{(0)}>u)=u^{-r}L(u)$, where $L$ is slowly varying.
 If
 \begin{equation}
  \label{T4.1}
  \liminf_{a\rightarrow 0} \frac{x_a}{a^{-1}\log{a^{-1}}} > e^r\frac{(r-2)}{2}\sigma^2,
 \end{equation}
 then
 \begin{equation}
  \label{T4.2}
  \mathbf{P}(M^{(a)}>x_a) \sim \frac{x_a^{-r+1}L(x_a)}{(r-1)a} \quad \text{as } a\rightarrow 0.
 \end{equation}
\end{theorem}
Olvera-Cravioto, Blanchet and Glynn \cite{OBG} have shown that for an $M/G/1$ queue the relation (\ref{T4.2}) holds under the condition
$$
\liminf_{a\rightarrow 0} \frac{x_a}{a^{-1}\log{a^{-1}}} > \frac{(r-2)}{2}\sigma^2.
$$
We believe that the latter should be sufficient for the validity of (\ref{T4.2}) also in the general case. The extra factor
$e^r$ appears in (\ref{T4.1}) only as a consequence of the technique we used, and can be removed by adoption of (\ref{T3.2}) to the heavy-traffic setting.
\begin{theorem}
\label{T5}
 Assume that $\mathbf{E}[(\min\{0,X^{(0)}\})^2]<\infty$ and $\mathbf{P}(X^{(0)}>u)=u^{-r}L(u)$ with $r\in(1,2)$.
 If
 \begin{equation}
  \label{T5.1}
  \liminf_{a\rightarrow 0} g(a)x_a=\infty,
 \end{equation} 
 then
 \begin{equation}
  \label{T5.2}
  \mathbf{P}(M^{(a)}>x_a) \sim \frac{x_a^{-r+1}L(x_a)}{(r-1)a} \quad \text{as } a\rightarrow 0.
 \end{equation}
\end{theorem}
We have imposed the condition $\mathbf{E}[(\min\{0,X^{(0)}\})^2]<\infty$ just to use the Lorden inequality for the overshoot. 
If one replaces that condition by $\mathbf{E}[(\min\{0,X^{(0)}\})^t]<\infty$ with $t\in(1,2)$, then, using Proposition \ref{P1}, one
can show that (\ref{T5.2}) holds for $x_a\gg a^{-t/(t-1)^2}$ only. The reason is the roughness of Proposition \ref{P1} for small
values  of $a$. Indeed, if we use (\ref{P1.1}) even with $t=2$, then we get the bound $\mathbf{E}[R_z]\leq Ca^{-2}$, which is much
worse then the Lorden inequality.


\section{Proofs}
\subsection{Proofs of Theorems \ref{T1} and \ref{T2}}
We set for brevity $\tau=\tau_z$.
\begin{lemma}
\label{L1}
For all $h$ satisfying 
\begin{equation}
\label{h-cond}
\mathbf{E}[e^{hX},X\leq y]\leq1
\end{equation}
we have the inequality
\begin{equation}
\label{L1.0}
\mathbf{P}(M_\tau>x)\leq \Bigl(1+\frac{1}{e^{hx}-1}\Bigr)\mathbf{E}[\tau]\mathbf{P}(X>y)
+\mathbf{E}[\tau]\frac{ah}{e^{hx}-1}.
\end{equation}
\end{lemma}
\begin{proof}
Our strategy is to truncate the random variables $X_i$ in the level $y$:
\begin{align}
	\nonumber
	\mathbf{P}\left(M_\tau > x\right) 
	&\le \mathbf{P}\left(M_\tau > x, \max_{1 \le k \le \tau} X_{k} \le y\right) 
	+ \mathbf{P}\left(\max_{1 \le k \le \tau} X_{k} > y\right)\\
	&= \mathbf{P}\left(M_\tau \mathbf{1}_{\{\max_{1 \le k \le \tau} X_{k} \le y\}} > x \right)
	+ \mathbf{P}\left(\max_{1 \le k \le \tau} X_{k} > y\right). \label{eq2.1}
\end{align}
From the Wald identity follows
\begin{equation}\label{L1.1}
\mathbf{P}\left(\max_{1 \le k \le \tau} X_{k} > y\right)
\le \mathbf{E}\left[\sum^{\tau}_{k=1} \mathbf{1}_{\{X_{k}>y\}}\right]
= \mathbf{E}\left[\tau\right]\mathbf{P}\left(X>y\right).
\end{equation}
To examine the first term on the right-hand-side of (\ref{eq2.1}) we introduce the process $\{W_k\}$ defined by
$$
W_0:=1, \quad W_k:=\prod_{i=1}^k e^{h X_i} \mathbf{1}_{\{X_i\le y\}}, \ k\ge 1.
$$
It is clear that if $h$ satisfies (\ref{h-cond}), $\{W_k\}$ is a positive supermartingale. Define 
$$
\sigma_y:=\min\{k\ge 1:X_k>y\},\ t_x:=\min\{k\ge 1:S_k>x\} \ \text{and} \ T:=\min\{\sigma_y,t_x,\tau\}.
$$
Applying the Optional Stopping Theorem to the supermartingale $\{W_{k\wedge T}\}$ leads us
$$
 1= W_0\ge \mathbf{E}[W_T]
 = \mathbf{E}\left[W_{T} \mathbf{1}_{\{t_{x} < \tau, t_{x} < \sigma_{y}\} }\right] 
  + \mathbf{E}\left[W_{T} \mathbf{1}_{\{ \tau < t_{x}, \tau < \sigma_{y}\}}\right]. 
$$
We analyse the two terms on the right-hand-side separately:
$$
\mathbf{E}\left[W_{T} \mathbf{1}_{\{t_{x} < \tau, t_{x} < \sigma_{y}\} }\right]
\ge e^{h x}\mathbf{P}(t_x<\tau<\sigma_y)
= e^{h x}\mathbf{P}\left(M_\tau\mathbf{1}_{\{\max_{1 \le k \le \tau} X_{k} \le y\}} > x\right)
$$
and
\begin{align*}
	&\mathbf{E}\left[W_{T} \mathbf{1}_{\{ \tau < t_{x}, \tau < \sigma_{y}\}}\right] 
	= \mathbf{E}\left[e^{h S_{\tau}}	\right]
			-\mathbf{E}\left[e^{h S_{\tau}} \mathbf{1}_{\{M_\tau> x\}
			 \cup \{\max_{1 \le k \le \tau}X_{k} > y\}}\right]\\
	&\hspace{0.5cm}\ge \mathbf{E}\left[e^{h S_{\tau}}	\right] 
			- e^{-h z}\left(\mathbf{P}\left(M_\tau 
			\mathbf{1}_{\{\max_{1 \le k \le \tau} X_{k} \le y\}} > x\right)
			+\mathbf{P}\bigg(\max\limits_{1 \le k \le \tau}X_{k} > y\bigg)\right).
\end{align*}
Consequently,
$$
	\mathbf{P}\left(M_\tau \mathbf{1}_{\{\max_{1 \le k \le \tau} X_{k} \le y\}} > x\right)
	\le \frac{1-\mathbf{E}\left[e^{h S_{\tau}}\right]
		+\mathbf{P}\left(\max\limits_{1 \le k \le \tau}X_{k} > y\right)}{e^{h x}-1}
$$
and hence by applying (\ref{L1.1}),
$$
	\mathbf{P}\left(M_\tau \mathbf{1}_{\{\max_{1 \le k \le \tau} X_{k} \le y\}} > x \right) \le
		\frac{1-\mathbf{E}\left[e^{h S_{\tau}}\right]+\mathbf{E}[\tau]
		\mathbf{P}\left(X > y\right)}{e^{h x}-1}.
$$
It is easy to see that
$$
	\mathbf{E}\left[e^{h S_{\tau}}\right] \ge
	 \mathbf{E}\left[1+h S_{\tau}\right]= 1+h\mathbf{E}[S_\tau]
$$
and as a result we have
\begin{equation}\label{L1.2}
\mathbf{P}\left(M_\tau \mathbf{1}_{\{\max_{1 \le k \le \tau} X_{k} \le y\}} > x \right) \le \mathbf{E}[\tau]
		\frac{ha+
		\mathbf{P}\left(X > y\right)}{e^{h x}-1}.
\end{equation}
Applying (\ref{L1.1}) and (\ref{L1.2}) to the summands in (\ref{eq2.1}) finishes the proof.
\end{proof}

To prove Theorems \ref{T1} and \ref{T2} we need to choose a specific $h$ for which (\ref{h-cond}) holds. 
The optimal choice would be the positive solution of the equation $\mathbf{E}[e^{hX},X\leq y]=1$, 
which is in the spirit of the Cramer-Lundberg condition. But it is not clear how to solve this equation.
For this reason we replace $\mathbf{E}[e^{hX},X\leq y]=1$ by the equation $\phi(h,y)=1$, where 
$\phi(h,y)$ is an appropriate upper bound for $\mathbf{E}[e^{hX},X\leq y]$.

If $A_t<\infty$, we may use a bound from the proof of Theorem 2 from \cite{FN71}, which says
\begin{equation}\label{T1.2}
\mathbf{E}[e^{hX},X\leq y]\leq 1+h\mathbf{E}[X,|X|\leq y]+\frac{e^{hy}-1-hy}{y^{t}}A_t.
\end{equation}
Using the Markov inequality we also obtain
$$
\mathbf{E}[X,|X|\leq y]\le -a-\mathbf{E}[X,X\leq -y]\leq -a+\frac{A_t}{y^{t-1}},
$$
and therefore
\begin{align*}
\mathbf{E}[e^{hX},X\leq y]&\leq 1-ha+\frac{e^{hy}-1}{y^{t}}A_t.
\end{align*}
Put $h_0:=\frac{1}{y}\log\left(1+ay^{t-1}/A_t\right)$.
It is easy to see that 
$$
-h_0 a+\frac{e^{h_0y}-1}{y^{t}}A_t\le 0
$$ 
for all $y$ such that $y^{t-1}\ge(e-1)A_t/a$ and this implies that $h_0$ satisfies (\ref{h-cond}).
Using (\ref{L1.0}) with $h=h_0$ and applying the inequality 
$$
(1+u)^{x/y}\geq 1+u^{x/y}, \ x\geq y,
$$ 
we obtain
\begin{align*}
\mathbf{P}(M_\tau>x)
&\leq \frac{A_t^{x/y}}{a^{x/y-1}}\mathbf{E}[\tau] y^{-1-(t-1)x/y}\log\left(1+ay^{t-1}/A_t\right)\\
&\hspace{1cm}+\left(1+\frac{A_t^{x/y}}{a^{x/y}}y^{-(t-1)x/y}\right)\mathbf{E}[\tau]\mathbf{P}(X>y).
\end{align*}
Thus, the proof of Theorem \ref{T1} is complete.

In order to show that one can replace $\mathbf{E}[X]$ and $A_t$ by the corresponding truncated moments, see Remark~\ref{R1}, 
we first note that analogously to (\ref{T1.2}) and by additionally using $e^x-1\le xe^x$, 
$$
\mathbf{E}[e^{hX},X\leq y]\leq 1+h\mathbf{E}[X,|X|\leq y]+h\frac{e^{hy}-1}{y^{t-1}}\mathbf{E}[|X|^t,|X|\leq y].
$$
If $\mathbf{E}[X,|X|\leq y]<0$, then 
$$
h_0:=\frac{1}{y}\log\left(1+\frac{|\mathbf{E}[X,|X|\leq y]|y^{t-1}}{\mathbf{E}[|X|^t,|X|\leq y]}\right)
$$
is strictly positive and solves
$$
h\mathbf{E}[X,|X|\leq y]+h\frac{e^{hy}-1}{y^{t-1}}\mathbf{E}[|X|^t,|X|\leq y]=0.
$$
Therefore, we may use Lemma \ref{L1} with $h=h_0$ and get an inequality with truncated moments.

To bound $\mathbf{E}[e^{hX},X\leq y]$ under the conditions of Theorem \ref{T2}
we proceed similar to the proof of Theorem 3 from \cite{N79} and get
$$
\mathbf{E}[e^{hX},X\leq y]\leq 1-ha+e^t\mathbf{Var}(X)\frac{h^2}{2}+\frac{e^{hy}-1-hy}{y^t}A_{t,+}.
$$
Following further the method from the proof of this Theorem, we split this upper bound into two parts:
$$
-\alpha ha+e^t\mathbf{Var}(X)\frac{h^2}{2}=:f_1(h),
$$ 
$$
-\beta ha+\frac{e^{hy}-1-hy}{y^t}A_{t,+}=:f_2(h).
$$
We consider $f_1$ and  $f_2$ separately. It is clear that
$$
h_1:=\frac{2\alpha a}{e^t \mathbf{Var}(X)}
$$ 
is the positive solution of the equation $f_1(h)=0$. Moreover, $f_1(h)<0$ for all $h\in(0,h_1)$.

Furthermore, it is easy to see that $f_2$ takes it's unique minimum in 
$$
h_2:=\frac{1}{y}\log\left(1+a\frac{\beta y^{t-1}}{A_{t,+}}\right).
$$ 
Since $f_2$ is convex, one has
\begin{equation}\label{ineq}
f_2(h)<0\quad\text{for all }h\in(0,h_2].
\end{equation}

The assumption in Theorem \ref{T2}(i) means that $h_1\leq h_2$. In this case, taking into account (\ref{ineq}), we obtain
$$
f_1(h_1)+f_2(h_1)<0.
$$
From the latter inequality we conclude that
$h_1$ satisfies (\ref{h-cond}) and by applying (\ref{L1.0}) with $h=h_1$ we obtain (\ref{T2.1}).

Under the conditions of Theorem \ref{T2} (ii) we have $h_2\leq h_1$. By the same arguments we get
$$
f_1(h_2)+f_2(h_2)<0.
$$
Then, applying (\ref{L1.0}) with $h=h_2$ and using the inequality $(1+u)^{x/y}\geq u^{x/y}$, we obtain (\ref{T2.2}).
\subsection{Proof of Proposition \ref{P1}} 

We want to use Theorem 2.1 from \cite{BF98}. If we put $F:=F_{-X}$ the conditions (G1)-(G3) of this Theorem are fulfilled in our setting.
Hence we get 
\begin{equation}
\label{eq2.2}
\mathbf{E}[R_z] 
\le c\int_0^\infty\mathbf{P}(-X>u)du + c\int_0^\infty\int_u^{u+z}\mathbf{P}(-X>v)dvdu,
\end{equation}
where 
\begin{equation}
\label{c}
c = \frac{b^*(\epsilon a)}{a(1-\epsilon)}
\end{equation}
with $b^*(u)=\min\{v:-\mathbf{E}[X,X<-v]\le u\}$ and $\epsilon \in (0,1)$ arbitrary.
By the Theorem of Fubini we obtain
\begin{equation}
\label{eq2.3}
\int_0^\infty\mathbf{P}(-X>u)du = \mathbf{E}[-X,X<0].
\end{equation}
Changing the order of integration gives us
\begin{align}
\nonumber
\int_0^\infty\int_u^{u+z}\mathbf{P}(-X>v)dvdu
= \int_0^z v\mathbf{P}(-X>v)dv + z\int_z^\infty \mathbf{P}(-X>v)dv\\
\le z^{2-t} \int_0^\infty v^{t-1}\mathbf{P}(-X>v)dv
= \frac{z^{2-t}}{t} A_{t,-}. \label{eq2.4}
\end{align}
As you can easily see,
$$
b^*(u)\le \left(\frac{A_{t,-}}{u}\right)^{1/(t-1)},
$$
therefore by (\ref{c})
$$
c \le \frac{A_{t,-}^{1/(t-1)}}{a^{t/(t-1)}\epsilon^{1/(1-t)}(1-\epsilon)},
$$
and by minimisation over $\epsilon \in (0,1)$
\begin{equation}
\label{c2}
c \le \frac{t^{t/(t-1)} A_{t,-}^{1/(t-1)}}{(t-1)a^{t/(t-1)}}.
\end{equation}
Finally, combining (\ref{eq2.2}), (\ref{eq2.3}), (\ref{eq2.4}) and (\ref{c2}) gives us the desired result.

\subsection{Proof of Theorem \ref{T3}}
We prove (\ref{T3.1}) only. The proof of the second bound goes along the same line.

Using Theorem \ref{T1} with $y=\theta(x+jz)$, we obtain
\begin{align*}
\nonumber
\mathbf{P}(M_\tau>x+jz)&\leq 
\frac{A_t^{1/\theta}\theta^{-1-(t-1)/\theta}\mathbf{E}[\tau_z]}{a^{1/\theta-1}(x+jz)^{1+(t-1)/\theta}}
\log\left(1+\frac{a\theta^{t-1}(x+jz)^{t-1}}{A_t}\right)\\
&+\left(1+\frac{A_t^{1/\theta}\theta^{-(t-1)/\theta}}{a^{1/\theta}}(x+jz)^{-(t-1)/\theta}\right)
\mathbf{E}[\tau_z]\mathbf{P}(X>\theta(x+jz)),
\end{align*}
and in view of (\ref{eq1.1}),
\begin{align}\label{T3.3}
\nonumber
\mathbf{P}(M>x)&\leq 
\frac{A_t^{1/\theta}\theta^{-1-(t-1)/\theta}}{a^{1/\theta-1}}
\mathbf{E}[\tau_z]\Sigma_1(x,z)\\
&\hspace{1cm}+\left(1+\frac{A_t^{1/\theta}}{a^{1/\theta}}x^{-(t-1)/\theta}\right)
\mathbf{E}[\tau_z]\left(\mathbf{P}(X>\theta x)+\Sigma_2(x,z)\right),
\end{align}
where
$$
\Sigma_1(x,z):=\sum_{j=0}^{\infty}\log\left(1+\frac{a\theta^{t-1}(x+jz)^{t-1}}{A_t}\right)
(x+jz)^{-1-(t-1)/\theta}
$$
and
$$
\Sigma_2(x,z):=\sum_{j=1}^{\infty}\mathbf{P}(X>\theta(x+jz)).
$$
Define 
$$
\tilde{\Sigma}_1(x,z):=\sum_{j=1}^{\infty}\log\left(1+\frac{a\theta^{t-1}(x+jz)^{t-1}}{A_t}\right)
(x+jz)^{-1-(t-1)/\theta}.
$$
The summands in this sum are strictly decreasing, so we conclude by the integral criteria for sums:
\begin{align*}
\tilde{\Sigma}_1(x,z)
\le &\sum_{j=1}^{\infty}\int_{j-1}^{j}\log\left(1+\frac{a\theta^{t-1}(x+uz)^{t-1}}{A_t}\right)
(x+uz)^{-1-(t-1)/\theta}du\\
= &\frac{1}{z} \int_x^\infty \log\left(1+\frac{a\theta^{t-1}w^{t-1}}{A_t}\right)
w^{-1-(t-1)/\theta} dw
\end{align*}
and further by integration by parts,
\begin{align*}
&\frac{1}{z}\int_x^\infty \log\left(1+\frac{a\theta^{t-1}w^{t-1}}{A_t}\right)
w^{-1-(t-1)/\theta} dw\\
&\hspace{1cm} \le \frac{\theta}{z(t-1)}\log\left(1+\frac{a\theta^{t-1}x^{t-1}}{A_t}\right)x^{-(t-1)/\theta}
+\frac{\theta^2}{z(t-1)}x^{-(t-1)/\theta}.
\end{align*}
Therefore, for all $x$ sufficing $x^{t-1}\ge\theta^{1-t}(e^{\theta}-1)A_t a^{-1}$ and $x\ge z(t-1)\theta^{-1}$,
$$
\Sigma_1(x,z)
\le \frac{3 \theta}{z(t-1)}\log\left(1+\frac{a\theta^{t-1}x^{t-1}}{A_t}\right) x^{-(t-1)/\theta}.
$$
Furthermore, it is easy to see that
\begin{equation}
\label{eq2.4.1}
\Sigma_2(x,z)\leq
\sum_{j=1}^{\infty}\int_{j-1}^{j}\mathbf{P}(X>\theta (x+u z))du=\frac{1}{\theta z}G(\theta x).
\end{equation}
and Theorem \ref{T3} is proved.

\subsection{Proof of Theorem \ref{T4}}
Foss, Korshunov and Zachary have shown, see Theorem 5.1 in \cite{FKZ}, that for any random walk with the drift $-a$ and $x_a$ with 
$x_a\to \infty$ as $a\to 0$ one has the following lower bound:
$$
 \liminf_{a \to 0} \frac{\mathbf{P}(M^{(a)}>x_a)}{a^{-1}G(x_a)} \ge 1.
$$
It follows from the regular variation of $\mathbf{P}(X^{(0)}>u)$, that
\begin{equation}
 \label{eq2.4.2}
 G(x_a) \sim \frac{1}{r-1}x_a^{-r+1}L(x_a) \quad \text{as } a\to \infty,
\end{equation}
therefore
$$
 \mathbf{P}(M^{(a)}\geq x_a) 
 \ge (1+o(1))\frac{x_a^{-r+1}L(x_a)}{(r-1)a} \quad \text{as } a \to 0. 
$$
Thus, we only have to derive an upper bound. During this proof we assume $a$ to be sufficiently small in every inequality. 
We want to apply Theorem \ref{T3} (ii) with $t<r$. 
It is clear that
$$
A_{t,+}^{(a)}:=\mathbf{E}[(X^{(a)})^t,X^{(a)}>0]\leq \mathbf{E}[(X^{(0)})^t,X^{(0)}>0],
$$
therefore $A_{t,+}^{(a)}$ is finite for $t<r$ and
$$
\lim_{a\to0}A_{t,+}^{(a)}=A_{t,+}^{(0)}>0.
$$
Further, we have to show that (\ref{T2.2.1}) is fulfilled for $y=\theta x$ under our assumptions.
Since the function $y^{-1}\log(1+\beta a y^{t-1}/A_{t,+}^{(a)})$ is decreasing for $y\gg a^{1/(t-1)}$, we have
the following bound for $x_a\geq ca^{-1}\log a^{-1}$:
\begin{align*}
\frac{1}{\theta x_a}\log\left(1+\frac{\beta \theta^{t-1}a x_a^{t-1}}{A_{t,+}^{(a)}}\right)
&\leq\frac{a}{\theta c\log a^{-1}}\log\left(1+\frac{\beta \theta^{t-1}c^{t-1}}{A_{t,+}^{(a)}}a^{2-t}\log^{t-1}a^{-1}\right)\\
&=\frac{(t-2)}{\theta c}a(1+o(1)).
\end{align*}
This implies that if $c>e^r(r-2)\sigma^2/2$ and $\theta=(t-2)/(r-2)$, we can choose $\alpha<1$ so close to $1$ that
$x_a$ satisfies (\ref{T2.2.1}).

We take $z=z_a$ satisfying $a^{-1}\ll z\ll x_a$. Then, combining (\ref{tau}) and (\ref{lorden}), we get
\begin{equation}\label{eq2.5}
\frac{\mathbf{E}[\tau_z]}{z}\sim\frac{1}{a}\quad\text{as }a\to0.
\end{equation}
Since $(t-1)/\theta=(t-1)(r-2)/(t-2)>r-1$, 
\begin{equation}
\label{eq2.6}
\beta^{-1/\theta} \frac{\mathbf{E}[\tau_z]}{z}\log\left(1+\frac{\beta \theta^{t-1} a x_a^{t-1}}{A_{t,+}^{(a)}}\right) x_a^{-(t-1)/\theta}=
o\left(a^{-1}x_{a}^{-r+1}L(x_a)\right).
\end{equation}
Further, it follows from the condition $z=o(x)$ and the regular variation of $\mathbf{P}(X^{(0)}>x)$ that
\begin{equation}
 \label{eq2.6.1} 
 z\mathbf{P}(X^{(a)}>x_a)=o\left(x_{a}^{-r+1}L(x_a)\right).
\end{equation}
Combining (\ref{eq2.5}) with (\ref{eq2.6.1}) and (\ref{eq2.4.2}), we obtain
\begin{align}
\label{eq2.7}
\nonumber
&\left(1+\left(\frac{A_{t,+}^{(a)}}{\beta\theta^{t-1}ax_a^{t-1}}\right)^{1/\theta}\right)\mathbf{E}[\tau_z]
\left(\frac{1}{\theta z}G(\theta x_a)+\mathbf{P}(X^{(a)}>\theta x_a)\right)\\
&\hspace{6cm} \sim \theta^{-r}(r-1)^{-1}a^{-1}x_a^{-r+1}L(x_a)
\end{align}
and plugging (\ref{eq2.6}) and (\ref{eq2.7}) into (\ref{T3.2}) gives us
$$
\limsup_{a\to0}\frac{\mathbf{P}(M^{(a)}>x_a)}{a^{-1}x_a^{-r+1}L(x_a)}\leq \theta^{-r}(r-1)^{-1}.
$$
To complete the proof it suffices to note, that we can choose $\theta$ arbitrary close to 1 by choosing $t$ close to $r$. This implies that
the previous inequality is valid even with $\theta=1$.
\subsection{Proof of Theorem \ref{T5}}
We again need an upper bound only. Let $a$ be sufficiently small during this proof.

It follows from the assumptions in the theorem that $S_n^{(0)}/c_n$ converges weakly to a stable law of index $r$. The sequence $c_n$ can be 
taken from the equation $c_n^{-r}L(c_n)=1/n$. It is known that the function $g(a)$ in the heavy-traffic approximation can be defined by the relations
$$
g(a)=1/c_{n_a}\text{ and }an_a\sim c_{n_a}.
$$
The latter can be rewritten as
$$
c_{n_a}\sim a\frac{(c_{n_a})^r}{L(c_{n_a})}.
$$
From this we infer that (\ref{T5.1}) is equivalent to 
\begin{equation}
\label{g}
  \frac{a x_a^{r-1}}{L(x_a)} \to \infty \quad \text{as } a \to 0.
\end{equation}
We want to apply Theorem \ref{T1} with $-\mathbf{E}[X^{(a)},|X^{(a)}|\leq y]$ and $A_2(y)$ instead of $a$ and $A_2$ respectively with $y=\theta x$.
According to Remark \ref{R1} we have to show that $\mathbf{E}[X^{(a)},|X^{(a)}|\leq \theta x_a]$ is negative. Using the Markov inequality, we have
$$
\mathbf{E}[X^{(a)},|X^{(a)}|\leq \theta x_a]\leq -a+(\theta x_a)^{-1}\mathbf{E}[(\min\{0,X^{(0)}\})^2].
$$
In view of (\ref{g}), $ax_a\to\infty$. Therefore,
$$
\mathbf{E}[X^{(a)},|X^{(a)}|\leq \theta x_a]\leq -a(1+o(1)).
$$
Furthermore,
\begin{equation}
 \label{ord_A}
 A_2(y) \sim \frac{r}{2-r}y^{2-r}L(y)
\end{equation}
and consequently by $-\mathbf{E}[X^{(a)},|X^{(a)}|\leq \theta x_a] \sim a$,
\begin{align}
 \nonumber
 &A_2^{1/\theta}(\theta x_a)\mathbf{E}[\tau_z] \frac{(-\mathbf{E}[X^{(a)},|X^{(a)}|\leq \theta x_a])^{1-1/\theta}}{\theta^{1+1/\theta}x_a^{1+1/\theta}}
\log\left(1-\frac{\theta x_a\mathbf{E}[X^{(a)},|X^{(a)}|\leq \theta x_a]  }{A_2(\theta x_a)}\right)\\
 \label{eq2.8}
 &\hspace{0.5cm}\leq (1+o(1)) k_1 \mathbf{E}[\tau_z]\mathbf{P}(X^{(a)}>x_a)\log\left(1+k_2 \frac{ax_a^{r-1}}{L(x_a)}\right)
  \left(\frac{a x_a^{r-1}}{L(x_a)}\right)^{-(1/\theta-1)}
\end{align}
with $k_1$ and $k_2$ appropriate. Then, (\ref{g}) implies that 
\begin{equation}
 \label{eq2.9}
 \log\left(1+c_2 \frac{ax_a^{r-1}}{L(x_a)}\right)\left(\frac{a x_a^{r-1}}{L(x_a)}\right)^{-(1/\theta-1)}
 =o(1).
\end{equation}
Further,
$$
 \frac{A_2^{1/\theta}(\theta x_a)}{a^{1/\theta}}\theta^{-1/\theta} x_a^{-1/\theta}
 \sim k_3\left(\frac{a x_a^{r-1}}{L(x_a)}\right)^{-1/\theta}
$$
with $k_3$ suitable and hence by (\ref{g}),
\begin{equation}
\label{eq2.10}
 \left(1+\frac{A_2^{1/\theta}(\theta x_a)}{a^{1/\theta}}\theta^{-1/\theta} x_a^{-1/\theta}\right)
 = (1+o(1)).
\end{equation}
Then, combining (\ref{eq2.8}), (\ref{eq2.9}) and (\ref{eq2.10}), Theorem \ref{T1} with $t=2$ and $y=\theta x_a$ gives us
$$
 \mathbf{P}(M_\tau^{(a)}>x_a)\le (1+o(1))\theta^{-r}\mathbf{E}[\tau_z]\mathbf{P}(X^{(a)}>x_a),
$$
where $\theta \in (0,1)$ is arbitrary. Hence by $\theta \to 1$,
$$
 \mathbf{P}(M_\tau^{(a)}>x_a)\le (1+o(1))\mathbf{E}[\tau_z]\mathbf{P}(X^{(a)}>x_a).
$$
By the summation formula (\ref{eq1.1}) we get a bound for the total maximum:
\begin{equation}
\label{eq2.11}
 \mathbf{P}(M^{(a)}>x_a) 
 \le (1+o(1))\mathbf{E}[\tau_z]\sum_{j=0}^\infty \mathbf{P}(X^{(a)}>x_a+j z).
\end{equation}
Combining (\ref{eq2.4.1}) and (\ref{eq2.4.2}) with $a^{-1} \ll z \ll x_a$ gives us
\begin{align*}
 \sum_{j=0}^\infty \mathbf{P}(X^{(a)}>x_a+j z)
 &\le (1+o(1))\left(x_a^{-r}L(x_a)+\frac{x_a^{-r+1}L(x_a)}{z(r-1)}\right)\\
 &\sim (1+o(1))\frac{x_a^{-r+1}L(x_a)}{z(r-1)}
\end{align*}
and regarding (\ref{eq2.5}) completes the proof.

\vspace{12pt}

{\bf Acknowledgement.} We are grateful to Denis Denisov for useful references.

\end{document}